%% file: HMv3.tex
\documentclass[10pt]{amsart}

\usepackage{times,amsfonts,amsmath,amstext,amsbsy,amssymb,
  amsopn,amsthm,upref,eucal, amscd,color,url,caption}
\usepackage[T1]{fontenc}

\usepackage[pdftex]{graphicx}

\usepackage{tikz}
\usetikzlibrary{arrows}

\usepackage{longtable}
\usepackage{array}

\newtheorem{theorem}{Theorem}[section]
\newtheorem{lemma}[theorem]{Lemma}

\numberwithin{equation}{section}

\theoremstyle{definition}

\newtheorem{remark}[theorem]{Remark}

 \def\Rset{\mathbb{R}}

\def\leq{\leqslant }
\def\geq{\geqslant}

\newcommand{\SL}{{\mathrm{SL}}}

\newcommand{\Teichmuller}{{Teich\-m\"uller }}

\begin{document}
\title[An origami of genus $3$ with arithmetic KZ monodromy]{An origami of genus $3$ with arithmetic Kontsevich--Zorich monodromy}

\author{Pascal Hubert}
\address{Pascal Hubert: Aix Marseille Universit\'e, CNRS, Centrale Marseille, Institut de Math\'ematiques de Marseille, I2M - UMR 7373, 13453 Marseille, France.}
\email{hubert.pascal@gmail.com.}

\author{Carlos Matheus}
\address{Carlos Matheus: Centre de Math\'ematiques Laurent Schwartz, CNRS (UMR 7640), \'Ecole Polytechnique, 91128 Palaiseau, France.}
\email{carlos.matheus@math.cnrs.fr}
\urladdr{http://carlos.matheus.perso.math.cnrs.fr/}

\date{\today}

\begin{abstract}
In this note, we exploit the arithmeticity criterion of Oh and Benoist--Miquel to exhibit an origami in the principal stratum of the moduli space of translation surfaces of genus three whose Kontsevich--Zorich monodromy is not thin in the sense of Sarnak. 
\end{abstract}
\maketitle

\setcounter{tocdepth}{1}
\tableofcontents

\section{Introduction}\label{s.introduction}

The dynamics of the action of $SL(2,\mathbb{R})$ on moduli spaces of translation surfaces is driven by the Kontsevich--Zorich monodromy consisting of the matrices encoding changes of basis in absolute homology of translation surfaces along $SL(2,\mathbb{R})$-orbits. 

The nature of the Kontsevich--Zorich monodromy depends heavily on the support of the ergodic $SL(2,\mathbb{R})$-invariant probability measure, and Sarnak asked how often a Kontsevich--Zorich monodromy is arithmetic or thin\footnote{Recall that a subgroup $\Gamma\subset GL_n(\mathbb{Z})$ with Zariski closure $G$ is called arithmetic, resp. thin, when the index of $\Gamma$ in the subgroup $G(\mathbb{Z})$ (of integral points of $G$) is finite, resp. infinite.} in his sense (compare with \S 3.2 of \cite{Sa}). 

In the case of Masur--Veech measures (of connected components of the strata of moduli spaces of translation surfaces), the corresponding Kontsevich--Zorich monodromies contain the Rauzy--Veech groups\footnote{Coming from a combinatorial process called Rauzy--Veech algorithm.}, and, as it turns out, the arithmeticity of Rauzy--Veech groups was recently established in \cite{AMY} and \cite{Gu}. In particular, the Kontsevich--Zorich monodromies associated to Masur--Veech measures are always arithmetic. 

In this article, we focus on the Kontsevich--Zorich monodromies of the natural measures supported on Teichm\"uller curves\footnote{Closed $SL(2,\mathbb{R})$-orbits in moduli spaces of translation surfaces.}. Any Teichm\"uller curve is known to be defined over a totally real number field, and we say that a Teichm\"uller curve is arithmetic if and only if it is defined over $\mathbb{Q}$. Equivalently, a Teichm\"uller curve is arithmetic if and only if it contains an origami / square-tiled surface\footnote{I.e., a translation surface obtained from a finite collection of squares of fixed sizes by gluing by translations pairs of parallel sides.}. 

In the moduli space of translation surfaces of genus $2$, it can be shown that the Kontsevich--Zorich monodromy of non-arithmetic, resp. arithmetic, Teichm\"uller curves are thin, resp. arithmetic (cf. \cite[\S3.2]{Sa}).\footnote{On the other hand, to the best of our knowledge, it seems that there were no available results concerning the arithmeticity or thinness of the Kontsevich--Zorich monodromy of Teichm\"uller curves in moduli spaces of translation surfaces of genus $g\geq 3$.}  

In the moduli space of translation surfaces of genus $3$, the main theorem of this note ensures the existence of arithmetic Kontsevich--Zorich monodromies associated to an arithmetic Teichm\"uller curves.

\begin{theorem}\label{t.HM} The non-tautological part\footnote{Here, the non-tautological part of the Kontsevich--Zorich monodromy of an origami $X$ means the following. The absolute homology of an origami $X$ admits a decomposition defined over $\mathbb{Z}$ into the direct sum of a tautological plane $H_1^{st}(X)$ and its symplectic orthogonal $H_1^{(0)}(X)$ (with respect to the intersection form). The Kontsevich--Zorich monodromy respects this decomposition and the non-tautological part of the Kontsevich--Zorich monodromy is its restriction to $H_1^{(0)}$. In particular, the non-tautological part of the Kontsevich--Zorich monodromy is a subgroup of $Sp(H_1^{(0)}(X))\simeq Sp(2g-2,\mathbb{Z})$, where $g$ is the genus of $X$.} of the Kontsevich--Zorich monodromy associated to a certain Teichm\"uller curve $\mathcal{C}$ generated by a certain origami $\mathcal{O}_1$ of genus $3$ is arithmetic. 
\end{theorem}

\begin{remark} It would be interesting to know whether the ``majority'' of non-tautological parts of Kontsevich--Zorich monodromies of origamis of genus $3$ is arithmetic: for instance, is it true that the Kontsevich--Zorich monodromies of all but finitely many origamis in the minimal stratum $\mathcal{H}(4)$ of the moduli space of translation surfaces of genus $3$ are arithmetic? 
\end{remark}

Closing this short introduction, let us describe the organization of this note. In Section \ref{s.example}, we describe the origami $\mathcal{O}_1$ and its Teichm\"uller curve $\mathcal{C}$. In Section \ref{sec:KZmonodromy}, we compute the Kontsevich--Zorich monodromy of $\mathcal{C}$: in particular, we describe two $4\times 4$ matrices (called $\rho(a)$ and $\rho(b)$ below) generating the non-tautological part of the Kontsevich--Zorich monodromy of $\mathcal{C}$. Finally, we rephrase Theorem \ref{t.HM} as Theorem \ref{t.HM'} below (for the sake of convenience), and we show that the desired arithmeticity statement can be deduced from a recent theorem of Benoist--Miquel \cite{BM} after some computations with certain powers of the two $4\times 4$ matrices introduced above. 

\begin{remark} In this note, we assume some familiarity with the basic features of origamis. In particular, the reader is invited to consult \S8 and Appendix C of the survey \cite{FM} for more details about the representation of origamis via permutations, the Veech and affine groups of origamis, etc. 
\end{remark}

\subsection*{Acknowledgements}

We are thankful to Alex Eskin and Vincent Delecroix for some discussions related to this note. 

\section{An arithmetic Teichm\"uller curve $\mathcal{C}$ with a single cusp} \label{s.example}

\subsection{The origami $\mathcal{O}_1$} Consider the square-tiled surface $\mathcal{O}_1$ associated to the pair of permutations 
$$h_{\mathcal{O}_1} = (1) (2,3,4,5)(6,7,8,9), \quad v_{\mathcal{O}_1} = (1,2,3,6)(4,7,9,8)(5)$$

The commutator $[h_{\mathcal{O}_1}, v_{\mathcal{O}_1}]:=v_{\mathcal{O}_1} h_{\mathcal{O}_1} v_{\mathcal{O}_1}^{-1} h_{\mathcal{O}_1}^{-1}$ is 
$$[h_{\mathcal{O}_1}, v_{\mathcal{O}_1}]=(1,9)(2,3)(4,6)(5,8)(7),$$
so that $\mathcal{O}_1\in\mathcal{H}(1,1,1,1)$ is a genus $3$ square-tiled surface. 

\begin{figure}[h!]
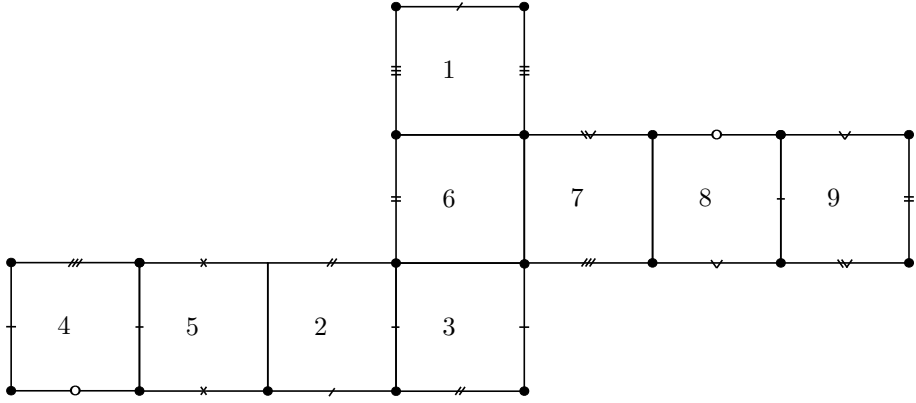\caption{Flat geometry of $\mathcal{O}_1$.}
\end{figure}

The $SL(2,\mathbb{Z})$-orbit of $\mathcal{O}_1$ consists of four elements. Indeed, this fact can be checked as follows. We recall that:
\begin{itemize} 
\item the generators $T=\left(\begin{array}{cc} 1 & 1 \\ 0 & 1 \end{array}\right)$ and $S=\left(\begin{array}{cc} 1 & 0 \\ 1 & 1 \end{array}\right)$ of $SL(2,\mathbb{Z})$ act on pairs of permutations $(h,v)$ by the rules $T(h,v) = (h,vh^{-1})$ and $S(h,v)=(hv^{-1}, v)$;
\item the pairs of permutations $(h,v)$ and $(\phi h \phi^{-1}, \phi v \phi^{-1})$ give rise to the same square-tiled surface. 
\end{itemize} 
Therefore, the $T$-orbit of $\mathcal{O}_1$ is $\{\mathcal{O}_1, \mathcal{O}_2, \mathcal{O}_3, \mathcal{O}_4\}$ where $\mathcal{O}_k:=T^k(\mathcal{O}_1)$ is given by the pair of permutations $(h_{\mathcal{O}_1}, v_{\mathcal{O}_k})$ with 
$$v_{\mathcal{O}_2} = (1,2,5,7)(3)(4,6,8,9), \quad v_{\mathcal{O}_3} = (1,2,7,8)(3,5,6,4)(9),$$   $$v_{\mathcal{O}_4} = (1,2,6,9)(3,7,4,5)(8)$$
As it turns out, the $T$-orbit of $\mathcal{O}_1$ accounts for its entire $SL(2,\mathbb{Z})$-orbit because 
$$S(\mathcal{O}_1)=(\phi_4^{-1} h_{\mathcal{O}_1} \phi_4, \phi_4^{-1} v_{\mathcal{O}_4} \phi_4)\simeq \mathcal{O}_4, \quad S^2(\mathcal{O}_1)=(\phi_3^{-1} h_{\mathcal{O}_1} \phi_3, \phi_3^{-1} v_{\mathcal{O}_3} \phi_3)\simeq \mathcal{O}_3,$$ 
$$S^3(\mathcal{O}_1)=(\phi_2^{-1} h_{\mathcal{O}_1} \phi_2, \phi_2^{-1} v_{\mathcal{O}_2} \phi_2)\simeq \mathcal{O}_2,$$
where 
$$\phi_4 = (1,6,2,9,4,3)(5,8)(7), \quad \phi_3 = (1,5,9,8)(2,6,3,4)(7)$$ 
and 
$$\phi_2 = (1,9)(2,4,5,3,6,8)(7).$$ 

\begin{remark}\label{r.-id} For later use, observe that the matrix $-\textrm{Id}$ acts on pairs of permutations by $-\textrm{Id}(h,v)=(h^{-1},v^{-1})$. In particular, the action of $-\textrm{Id}$ on $SL(2,\mathbb{Z})\cdot\mathcal{O}_1$ is completely described by the formulas 
$$-\textrm{Id}(\mathcal{O}_1) = (\psi_{3}^{-1} h_{\mathcal{O}_1} \psi_{3}, \psi_3^{-1} v_{\mathcal{O}_3} \psi_3)\simeq \mathcal{O}_3, \,\,\, -\textrm{Id}(\mathcal{O}_2) = (\psi_{4}^{-1} h_{\mathcal{O}_1} \psi_{4}, \psi_4^{-1} v_{\mathcal{O}_4} \psi_4)\simeq \mathcal{O}_4$$ 
where $\psi_3 := (1)(2,8,4,6)(3,7,5,9)$ and $\psi_4 := (1)(2,9,4,7)(3,8,5,6)$. 
\end{remark}

In summary, the $SL(2,\mathbb{Z})$-orbit of $\mathcal{O}_1$ can be depicted as in Figure \ref{f.2} below. 

\begin{figure}[h!]
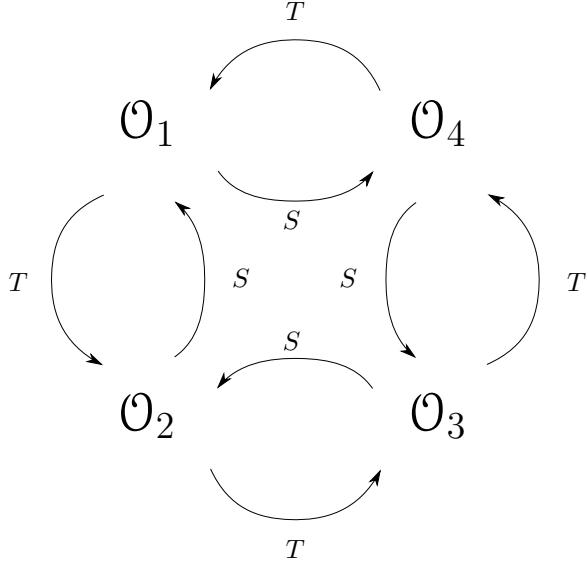\caption{$SL(2,\mathbb{Z})$-orbit of $\mathcal{O}_1$.}\label{f.2}
\end{figure}

It follows from this discussion that $SL(2,\mathbb{R})\cdot\mathcal{O}_1$ has a single cusp (i.e., single $T$-orbit in $SL(2,\mathbb{Z})\cdot\mathcal{O}_1$). 

\begin{remark}\label{r.homological-dim-O1} The homological dimension of $SL(2,\mathbb{R})\cdot\mathcal{O}_1$ in the sense of Forni \cite{Fo} is three. Thus, by the results in \cite{Fo}, the Lyapunov spectrum of the Kontsevich-Zorich cocycle over $SL(2,\mathbb{R})\cdot\mathcal{O}_1$ with respect to the Haar measure has the form 
$$1=\lambda_1>\lambda_2\geq\lambda_3>0>-\lambda_3\geq-\lambda_2>-\lambda_1=-1$$ 
Moreover, the Eskin-Kontsevich-Zorich formula \cite{EKZ} for the sum of non-negative Lyapunov exponents of the Kontsevich-Zorich cocycle implies that $1+\lambda_2+\lambda_3 = 2$, i.e., 
$$\lambda_2+\lambda_3=1$$
Finally, some \emph{numerical} experiments indicate that $\lambda_2\simeq 0.57...$ and $\lambda_3\simeq 0.43...$
\end{remark} 
 
\subsection{The stabilizer of $\mathcal{C}$} \label{Veech-group}

The group $\textrm{Aff}(\mathcal{O}_1)$ of affine homeomorphisms of $\mathcal{O}_1$ is the stabilizer of $\mathcal{C}:=SL(2,\mathbb{R})\cdot\mathcal{O}_1$ in the moduli space of translation surfaces. 

It is not hard to see that the subgroup $\textrm{Aut}(\mathcal{O}_1)\subset \textrm{Aff}(\mathcal{O}_1)$ of automorphisms of $\mathcal{O}_1$ is trivial. It follows that the elements of $\textrm{Aff}(\mathcal{O}_1)$ are determined by their linear parts in $SL(2,\mathbb{R})$, that is, the natural map 
$$\textrm{Aff}(\mathcal{O}_1)\to SL(2,\mathbb{R})$$
is injective. Hence, $\textrm{Aff}(\mathcal{O}_1)$ is isomorphic to its image $SL(\mathcal{O}_1)$ under this map. 

The group $SL(\mathcal{O}_1)$ is the finite-index subgroup of $SL(2,\mathbb{Z})$ consisting of all elements of $SL(2,\mathbb{R})$ stabilizing $\mathcal{O}_1$: in the literature, $SL(\mathcal{O}_1)$ is called the Veech group of $\mathcal{O}_1$.  

From Figure \ref{f.2} above, we see that $SL(\mathcal{O}_1)$ is an index four subgroup of $SL(2,\mathbb{Z})$. Furthermore, $SL(\mathcal{O}_1)$ is a congruence subgroup of level $4$, and the Teichm\"uller curve $\mathcal{C} =SL(2,\mathbb{R})/SL(\mathcal{O}_1)$ has genus zero. Thus, $SL(\mathcal{O}_1)$ is generated by elliptic and parabolic elements: indeed, one can check that $SL(\mathcal{O}_1)$ is generated by the following two elliptic matrices 
$$a:=\left(\begin{array}{cc} 0 & -1 \\ 1 & -1 \end{array}\right), \quad b:=\left(\begin{array}{cc} 1 & -3 \\ 1 & -2\end{array}\right)$$
of orders $3$. 

The group structure of $SL(\mathcal{O}_1)$ is provided by the following lemma: 

\begin{lemma}\label{l.Veech-group-structure} $SL(\mathcal{O}_1)$ is the free product 
$$SL(\mathcal{O}_1)=\langle a\rangle\ast\langle b\rangle \simeq \mathbb{Z}/3\mathbb{Z}\ast \mathbb{Z}/3\mathbb{Z}$$
\end{lemma}

\begin{proof} Consider the twelve cones $C_k\subset\mathbb{R}^2-\{(0,0)\}$ defined by the following properties:
\begin{itemize}
\item $C_{6+l} = -C_l$ for each $l=1,\dots, 6$;
\item each $C_l$, $l=1,\dots, 6$, consists of the convex combinations of positive multiples of the vectors $v_l$ and $v_{l+1}$, where $v_1:=(1,0)$, $v_2:=(2,1)$, $v_3:=(1,1)$, $v_4:=(1,2)$, $v_5:=(0,1)$, $v_6:=(-1,1)$ and $v_7:=(-1,0)$. 
\end{itemize} 

A simple calculation shows that 
\begin{itemize}
\item $a(v_l)=v_{l+4}$ for each $k=1,\dots, 6$;
\item $b(v_1)=v_3$, $b(v_2)=v_7$, $b(v_3)=-v_2$, $b(v_4)=(-5,-3)\in C_8$, $b(v_5)=(-3,-2)\in C_8$, $b(v_6)=(-4,-3)\in C_8$ and $b(v_7)=-v_3$. 
\end{itemize}

It follows that $\{a, a^2\}=\langle a\rangle\setminus\{\textrm{Id}\}$ and $\{b, b^2\}=\langle b\rangle \setminus \{\textrm{Id}\}$ play ping-pong with the tables 
$$X:=(C_1\cup C_2)\cup (C_7\cup C_8)$$
and 
$$Y:=C_3\cup C_4\cup C_5\cup C_6\cup C_9\cup C_{10}\cup C_{11}\cup C_{12}$$
in the sense that $X$ and $Y$ are disjoint subsets of $\mathbb{R}^2$ such that 
\begin{itemize}
\item $a(X)=(C_5\cup C_6)\cup(C_{11}\cup C_{12})\subset Y$, $a^2(X) = (C_9\cup C_{10})\cup (C_3\cup C_4)\subset Y$; 
\item $b(Y)\subset C_2\cup C_8\subset X$, $b^2(Y)\subset C_1\cup C_7\subset X$.
\end{itemize}

By the ping-pong lemma\footnote{Here, we are using the version of the ping-pong lemma stated as Theorem 2.1 in  Brav--Thomas article \cite{BT}.}, we conclude that $SL(\mathcal{O}_1) = \langle a\rangle\ast\langle b\rangle$.
\end{proof}

\begin{remark} The construction of these cones was inspired by Brav--Thomas paper \cite{BT}. 
\end{remark}

\section{The Kontsevich-Zorich monodromy of $\mathcal{C}$} \label{sec:KZmonodromy}

The representation $\alpha:\textrm{Aff}(\mathcal{O}_1)\to \textrm{Sp}(H_1(\mathcal{O}_1,\mathbb{Z}))$ is called Kontsevich-Zorich cocycle over $\mathcal{C}$. In the sequel, we will compute the image under $\alpha$ of the generators $a$ and $b$ of $SL(\mathcal{O}_1)\simeq\textrm{Aff}(\mathcal{O}_1)$. 

\subsection{The relative homology groups of $\mathcal{O}_k$, $k=1,\dots, 4$}  

Given $\mathcal{O}_k\in SL(2,\mathbb{Z})\cdot\mathcal{O}_1$, $k=1,\dots, 4$, let us denote by $\sigma_g^{(k)}$, resp., $\zeta_g^{(k)}$ the relative cycles on $\mathcal{O}_k$ consisting of the bottommost horizontal and leftmost vertical sides of the square numbered $g\in\{1,\dots,9\}$. 

Note that each square $g$ of $\mathcal{O}_k$ gives a relation $\sigma_g^{(k)} + \zeta_{h_{\mathcal{O}_1}(g)}^{(k)} = \zeta_g^{(k)} + \sigma_{v_{\mathcal{O}_k}(g)}^{(k)}$, that is,
\begin{itemize}
\item $\sigma_1^{(1)} = \sigma_2^{(1)}$, $\sigma_2^{(1)}+\zeta_3^{(1)} = \zeta_2^{(1)} + \sigma_3^{(1)}$, $\sigma_3^{(1)}+\zeta_4^{(1)} = \zeta_3^{(1)} + \sigma_6^{(1)}$, $\sigma_4^{(1)}+\zeta_5^{(1)} = \zeta_4^{(1)} + \sigma_7^{(1)}$, $\zeta_2^{(1)} = \zeta_5^{(1)}$, $\sigma_6^{(1)}+\zeta_7^{(1)} = \zeta_6^{(1)} + \sigma_1^{(1)}$, $\sigma_7^{(1)}+\zeta_8^{(1)} = \zeta_7^{(1)} + \sigma_9^{(1)}$, $\sigma_8^{(1)}+\zeta_9^{(1)} = \zeta_8^{(1)} + \sigma_4^{(1)}$, $\sigma_9^{(1)}+\zeta_6^{(1)} = \zeta_9^{(1)} + \sigma_8^{(1)}$;
\item $\sigma_1^{(2)} = \sigma_2^{(2)}$, $\sigma_2^{(2)}+\zeta_3^{(2)} = \zeta_2^{(2)} + \sigma_5^{(2)}$, $\zeta_3^{(2)} = \zeta_4^{(2)}$, $\sigma_4^{(2)}+\zeta_5^{(2)} = \zeta_4^{(2)} + \sigma_6^{(2)}$, $\sigma_5^{(2)}+\zeta_2^{(2)} = \zeta_5^{(2)} + \sigma_7^{(2)}$, $\sigma_6^{(2)}+\zeta_7^{(2)} = \zeta_6^{(2)}+ \sigma_8^{(2)}$, $\sigma_7^{(2)}+\zeta_8^{(2)} = \zeta_7^{(2)} + \sigma_1^{(2)}$, $\sigma_8^{(2)}+\zeta_9^{(2)} = \zeta_8^{(2)} + \sigma_9^{(2)}$, $\sigma_9^{(2)}+\zeta_6^{(2)} = \zeta_9^{(2)} + \sigma_4^{(2)}$;
\item $\sigma_1^{(3)} = \sigma_2^{(3)}$, $\sigma_2^{(3)}+\zeta_3^{(3)} = \zeta_2^{(3)} + \sigma_7^{(3)}$, $\sigma_3^{(3)}+\zeta_4^{(3)} = \zeta_3^{(3)} + \sigma_5^{(3)}$, $\sigma_4^{(3)}+\zeta_5^{(3)} = \zeta_4^{(3)} + \sigma_3^{(3)}$, $\sigma_5^{(3)}+\zeta_2^{(3)} = \zeta_5^{(3)} + \sigma_6^{(3)}$, $\sigma_6^{(3)}+\zeta_7^{(3)} = \zeta_6^{(3)} + \sigma_4^{(3)}$, $\sigma_7^{(3)}+\zeta_8^{(3)} = \zeta_7^{(3)} + \sigma_8^{(3)}$, $\sigma_8^{(3)}+\zeta_9^{(3)} = \zeta_8^{(3)} + \sigma_1^{(3)}$, $\zeta_6^{(3)} = \zeta_9^{(3)}$; 
\item $\sigma_1^{(4)} = \sigma_2^{(4)}$, $\sigma_2^{(4)}+\zeta_3^{(4)} = \zeta_2^{(4)} + \sigma_6^{(4)}$, $\sigma_3^{(4)}+\zeta_4^{(4)} = \zeta_3^{(4)} + \sigma_7^{(4)}$, $\sigma_4^{(4)}+\zeta_5^{(4)} = \zeta_4^{(4)} + \sigma_5^{(4)}$, $\sigma_5^{(4)}+\zeta_2^{(4)} = \zeta_5^{(4)} + \sigma_3^{(4)}$, $\sigma_6^{(4)}+\zeta_7^{(4)} = \zeta_6^{(4)} + \sigma_9^{(4)}$, $\sigma_7^{(4)}+\zeta_8^{(4)} = \zeta_7^{(4)} + \sigma_4^{(4)}$, $\zeta_8^{(4)} = \zeta_9^{(4)}$, $\sigma_9^{(4)}+\zeta_6^{(4)} = \zeta_9^{(4)} + \sigma_1^{(4)}$.
\end{itemize}

\subsection{The action of $SL(2,\mathbb{Z})$ on the relative homology groups}  

The matrix $T=\left(\begin{array}{cc} 1 & 1 \\ 0 & 1 \end{array}\right)$ takes $\mathcal{O}_k$ to $\mathcal{O}_{k+1}$, and it acts on the corresponding relative homology groups by the matrix $T_{k,k+1}$ such that 
$$T_{k,k+1}(\sigma_g^{(k)}) = \sigma_g^{(k+1)}, \quad T_{k,k+1}(\zeta_g^{(k)}) = \sigma_g^{(k+1)} + \zeta_{h_{\mathcal{O}_1}(g)}^{(k+1)},$$ 
Similarly, the matrix $S=\left(\begin{array}{cc} 1 & 0 \\ 1 & 1 \end{array}\right)$ takes $\mathcal{O}_{k}$ to $\mathcal{O}_{k-1}$, and it acts on the corresponding relative homology groups by the matrix $S_{k+1,k}$ such that 
$$S_{k+1,k}(\sigma_g^{(k+1)}) = \zeta_{\phi_k(g)}^{(k)} + \sigma_{v_{\mathcal{O}_k}(\phi_k(g))}^{(k)}, \quad S_{k+1,k}(\zeta_g^{(k+1)}) = \zeta_{\phi_k(g)}^{(k)}$$ 
Finally, $-\textrm{Id}$ exchange $\mathcal{O}_1$ and $\mathcal{O}_3$, resp. $\mathcal{O}_2$ and $\mathcal{O}_4$, and it acts on the corresponding relative homology groups by the matrices $(-\textrm{Id})_{1,3}=(-\textrm{Id})_{3,1}^{-1}$ and $(-\textrm{Id})_{2,4}=(-\textrm{Id})_{4,2}^{-1}$ such that 
$$(-\textrm{Id})_{1,3}(\sigma_g^{(1)}) = -\sigma_{v_{\mathcal{O}_3}(\psi_3(g))}^{(3)}, \quad (-\textrm{Id})_{1,3}(\zeta_g^{(1)}) = -\zeta_{h_{\mathcal{O}_3}(\psi_3(g))}^{(3)},$$
and 
$$(-\textrm{Id})_{2,4}(\sigma_g^{(2)}) = -\sigma_{v_{\mathcal{O}_4}(\psi_4(g))}^{(4)}, \quad (-\textrm{Id})_{2,4}(\zeta_g^{(2)}) = -\zeta_{h_{\mathcal{O}_4}(\psi_4(g))}^{(4)}$$

\subsection{The absolute homology groups of $\mathcal{O}_k$, $k=1,\dots, 4$} \label{ss.absolute-homology}  

The absolute homology group $H_1(\mathcal{O}_1,\mathbb{Q})$ has a basis $\mathcal{B}_k:=\{\Sigma_0^{(k)}, Z_0^{(k)}, \Sigma_1^{(k)}, \Sigma_2^{(k)}, Z_1^{(k)}, Z_2^{(k)}\}$ 
where 
$$\Sigma_0^{(k)} := \sum\limits_{g=1}^9 \sigma_g^{(k)}, \quad Z_0^{(k)} := \sum\limits_{g=1}^9 \zeta_g^{(k)},$$
$$\Sigma_1^{(k)} := \sum\limits_{j=1}^4\sigma_{h_{\mathcal{O}_1}^j(2)}^{(k)} - 4\sigma_1^{(k)}, \Sigma_2^{(k)} := \sum\limits_{j=1}^4\sigma_{h_{\mathcal{O}_1}^j(6)}^{(k)} - 4\sigma_1^{(k)},$$
and 
$$Z_1^{(1)} := \sum\limits_{j=1}^4 \zeta_{v_{\mathcal{O}_1}^j(1)}^{(1)} - 4\zeta_5^{(1)}, \quad Z_2^{(1)} := \sum\limits_{j=1}^4 \zeta_{v_{\mathcal{O}_1}^j(4)}^{(1)} - 4\zeta_5^{(1)},$$
$$Z_1^{(2)} := \sum\limits_{j=1}^4 \zeta_{v_{\mathcal{O}_1}^j(1)}^{(2)} - 4\zeta_3^{(2)}, \quad Z_2^{(2)} := \sum\limits_{j=1}^4 \zeta_{v_{\mathcal{O}_1}^j(4)}^{(1)} - 4\zeta_3^{(1)},$$
$$Z_1^{(3)} := \sum\limits_{j=1}^4 \zeta_{v_{\mathcal{O}_1}^j(1)}^{(1)} - 4\zeta_9^{(1)}, \quad Z_2^{(3)} := \sum\limits_{j=1}^4 \zeta_{v_{\mathcal{O}_1}^j(4)}^{(1)} - 4\zeta_9^{(1)},$$
$$Z_1^{(4)} := \sum\limits_{j=1}^4 \zeta_{v_{\mathcal{O}_1}^j(1)}^{(1)} - 4\zeta_8^{(1)}, \quad Z_2^{(4)} := \sum\limits_{j=1}^4 \zeta_{v_{\mathcal{O}_1}^j(4)}^{(1)} - 4\zeta_8^{(1)}.$$

Note that this basis is adapted to the decomposition $H_1(\mathcal{O}_k,\mathbb{Q}) = H_1^{st}(\mathcal{O}_k,\mathbb{Q}) \oplus H_1^{(0)}(\mathcal{O}_k,\mathbb{Q})$ in the sense that this decomposition corresponds to the partition $\mathcal{B}_k=\mathcal{B}_k^{st}\cup \mathcal{B}_k^{(0)}$ where $\mathcal{B}_k^{st}=\{\Sigma_0^{(k)}, Z_0^{(k)}\}$ and $\mathcal{B}_k^{(0)}=\mathcal{B}_k \setminus \mathcal{B}_k^{st}$, i.e., 
$$H_1^{st}(\mathcal{O}_k,\mathbb{Q}) = \mathbb{Q}\Sigma_0^{(k)}\oplus\mathbb{Q} Z_0^{(k)}$$
and 
$$H_1^{(0)}(\mathcal{O}_k,\mathbb{Q}) = \mathbb{Q}\Sigma_1^{(k)}\oplus\mathbb{Q} Z_1^{(k)}\oplus \mathbb{Q}\Sigma_2^{(k)}\oplus\mathbb{Q} Z_2^{(k)}.$$

Moreover, it is worth to point out that the matrix of the restriction to $H_1^{(0)}(\mathcal{O}_1,\mathbb{Z})$ of the intersection form $\Omega$ in the basis $\mathcal{B}_1^{(0)}$ is 
$$\Omega = \left(\begin{array}{cccc} 0 & 0 & -6 & -3 \\ 0 & 0 & -3 & 3 \\ 6 & 3 & 0 & 0 \\ 3 & -3 & 0 & 0 \end{array}\right)$$

\subsection{The action of $\textrm{Aff}(\mathcal{O}_1)$ on the absolute homology group}  

The formulas from the previous two subsections say that the matrices of $T_{k,k+1}$, $S_{k+1,k}$ and $-(\textrm{Id})_{k,k+2}$ with respect to the bases $\mathcal{B}_l$ are 
$$T_{1,2}=\left(\begin{array}{cccccc} 1 & 1 & 0 & 0 & 0 & 0 \\ 0 & 1 & 0 & 0 & 0 & 0 \\ 0 & 0 & 1 & 0 & 1 & 0 \\ 0 & 0 & 0 & 1 & 0 & 1 \\ 0 & 0 & 0 & 0 & 1 & 0 \\ 0 & 0 & 0 & 0 & 0 & 1\end{array}\right), \quad  T_{2,3}=\left(\begin{array}{cccccc} 1 & 1 & 0 & 0 & 0 & 0 \\ 0 & 1 & 0 & 0 & 0 & 0 \\ 0 & 0 & 1 & 0 & -1 & -1 \\ 0 & 0 & 0 & 1 & 0 & 1 \\ 0 & 0 & 0 & 0 & 1 & 0 \\ 0 & 0 & 0 & 0 & -1 & -1\end{array}\right),$$ 
$$T_{3,4}=\left(\begin{array}{cccccc} 1 & 1 & 0 & 0 & 0 & 0 \\ 0 & 1 & 0 & 0 & 0 & 0 \\ 0 & 0 & 1 & 0 & 0 & 1 \\ 0 & 0 & 0 & 1 & 1 & 0 \\ 0 & 0 & 0 & 0 & 1 & 0 \\ 0 & 0 & 0 & 0 & 0 & 1\end{array}\right), \quad  T_{4,1}=\left(\begin{array}{cccccc} 1 & 1 & 0 & 0 & 0 & 0 \\ 0 & 1 & 0 & 0 & 0 & 0 \\ 0 & 0 & 1 & 0 & 0 & 1 \\ 0 & 0 & 0 & 1 & -1 & -1 \\ 0 & 0 & 0 & 0 & 1 & 0 \\ 0 & 0 & 0 & 0 & -1 & -1\end{array}\right),$$ 
$$S_{1,4}=\left(\begin{array}{cccccc} 1 & 0 & 0 & 0 & 0 & 0 \\ 1 & 1 & 0 & 0 & 0 & 0 \\ 0 & 0 & 0 & 1 & 0 & 0 \\ 0 & 0 & 1 & 0 & 0 & 0 \\ 0 & 0 & 1 & 0 & 1 & 0 \\ 0 & 0 & 0 & 1 & 0 & 1\end{array}\right), \quad S_{4,3}=\left(\begin{array}{cccccc} 1 & 0 & 0 & 0 & 0 & 0 \\ 1 & 1 & 0 & 0 & 0 & 0 \\ 0 & 0 & -1 & -1 & 0 & 0 \\ 0 & 0 & 0 & 1 & 0 & 0 \\ 0 & 0 & 1 & 0 & 0 & 1 \\ 0 & 0 & -1 & -1 & 1 & 0\end{array}\right),$$ 
$$S_{3,2}=\left(\begin{array}{cccccc} 1 & 0 & 0 & 0 & 0 & 0 \\ 1 & 1 & 0 & 0 & 0 & 0 \\ 0 & 0 & 0 & 1 & 0 & 0 \\ 0 & 0 & 1 & 0 & 0 & 0 \\ 0 & 0 & 0 & 1 & 1 & 0 \\ 0 & 0 & 1 & 0 & 0 & 1\end{array}\right), \quad S_{2,1}=\left(\begin{array}{cccccc} 1 & 0 & 0 & 0 & 0 & 0 \\ 1 & 1 & 0 & 0 & 0 & 0 \\ 0 & 0 & 1 & 0 & 0 & 0 \\ 0 & 0 & -1 & -1 & 0 & 0 \\ 0 & 0 & 0 & 1 & 0 & 1 \\ 0 & 0 & -1 & -1 & 1 & 0\end{array}\right),$$ 
$$(-\textrm{Id})_{1,3} = \left(\begin{array}{cccccc} -1 & 0 & 0 & 0 & 0 & 0 \\ 0 & -1 & 0 & 0 & 0 & 0 \\ 0 & 0 & 0 & -1 & 0 & 0 \\ 0 & 0 & -1 & 0 & 0 & 0 \\ 0 & 0 & 0 & 0 & -1 & 0 \\ 0 & 0 & 0 & 0 & 0 & -1\end{array}\right) = (-\textrm{Id})_{2,4}$$

This allows us to compute the images $\alpha(a)$ and $\alpha(b)$ of the generators $a$ and $b$ of $SL(\mathcal{O}_1)\simeq\textrm{Aff}(\mathcal{O}_1)$ under the KZ cocycle $\alpha:\textrm{Aff}(\mathcal{O}_1)\to \textrm{Sp}(H_1(\mathcal{O}_1,\mathbb{Z}))$. Indeed, 
$$a=\left(\begin{array}{cc} 0 & -1 \\ 1 & -1 \end{array}\right) = (-\textrm{Id}) T S^{-1}, \quad b=\left(\begin{array}{cc} 1 & -3 \\ 1 & -2\end{array}\right) = ST^{-3},$$
so that 
$$\alpha(a) = (-\textrm{Id})_{3,1} T_{2,3} S_{2,1}^{-1}, \quad \alpha(b) = S_{2,1}T_{2,3}^{-1}T_{3,4}^{-1}T_{4,1}^{-1}$$

For later use, we observe that these formulas give that the non-tautological subrepresentation $\rho:\textrm{Aff}(\mathcal{O}_1)\to \textrm{Sp}(H_1^{(0)}(\mathcal{O}_1,\mathbb{Z}))$ of $\alpha$ takes values 
$$\rho(a) = \left(\begin{array}{cccc} 0 & 0 & -1 & 0 \\ 0 & 0 & 1 & 1 \\ 
 0 & 1 & 0 & -1 \\ 1 & 0 & 1 & 1 \end{array}\right), \quad \rho(b) = \left(\begin{array}{cccc} 
1 & 0 & 3 & 3 \\ -1 & -1 & -2 & -1 \\ 0 & 1 & -1 & -1 \\ -1 & -1 & -1 & -1\end{array}\right)$$
(with respect to the basis $\mathcal{B}_1^{(0)}$ of $H_1^{(0)}(\mathcal{O}_1,\mathbb{Z})$) at the two generators $a$ and $b$ of $SL(\mathcal{O}_1)$. Moreover, if we denote by $p_1 = ST^{-4}S^{-1}T^4$, $p_2 = ST^{-4}ST^6\in SL(\mathcal{O}_1)$, then the characteristic polynomials $\chi_{p_1}(x)$ and $\chi_{p_2}(x)$ of the matrices $\rho(p_1)$ and $\rho(p_2)$ are 
$$\chi_{p_1}(x) = x^4 - 11x^3 + 29x^2 - 11x + 1$$ 
and 
$$\chi_{p_2}(x) = x^4 -2x^3 -16x^2 -2x +1$$

\section{Arithmeticity of the Kontsevich-Zorich group associated to $\mathcal{C}$} \label{sec:arithKZ}

This section is devoted to the study of the image of the representation $\rho:\textrm{Aff}(\mathcal{O}_1)\to \textrm{Sp}(H_1^{(0)}(\mathcal{O}_1,\mathbb{Z}))$ describing the non-tautological part of the Kontsevich--Zorich cocycle. 

\subsection{Zariski density of $\rho(\text{Aff}(\mathcal{O}_1))$ in $\text{Sp}(H_1^{(0)}(\mathcal{O}_1,\mathbb{R}))$}\label{ss.Zariski}   

The matrices $\rho(p_1)$ and $\rho(p_2)$ are Galois-pinching\footnote{Recall that a matrix $A\in Sp(2d,\mathbb{Z})$ is Galois-pinching whenever its eigenvalues are real and its characteristic polynomial is an irreducible polynomial over $\mathbb{Q}$ with largest possible Galois group (of order $2^d d!$).} in the sense of the article \cite{MMY} and the splitting fields of their characteristic polynomials are disjoint. 

Indeed, these facts follow from the analysis of the discriminants 
$$\Delta_1(\chi_{p_1}) = (-11)^2 - 4\times (29-2) = 13, \quad \Delta_1(\chi_{p_2}) = (-2)^2 - 4\times (-16-2) = 2^2\times 19$$
and 
$$\Delta_2(\chi_{p_1}) = (29+2)^2 - 4\times (-11)^2 = 3^2\times 53, \,\, \Delta_2(\chi_{p_2}) = (-16+2)^2 - 4\times(-2)^2 = 6^2\times 5$$ 
(cf. \cite[\S 6.7]{MMY}). 

By the Zariski density criterion of Prasad--Rapinchuk \cite[Theorem 9.10]{PR} (see also \cite[Theorem 1.5]{Ri}), we have that $\rho(\textrm{Aff}(\mathcal{O}_1))$ is Zariski-dense in $\textrm{Sp}(H_1^{(0)}(\mathcal{O}_1,\mathbb{R}))$. 

\begin{remark} The Zariski-denseness of $\rho(\textrm{Aff}(\mathcal{O}_1))$ allows to apply the main result of \cite{EsMat} in order to deduce that the Lyapunov spectrum of $\mathcal{C}$ is simple, i.e., 
$$1=\lambda_1>\lambda_2>\lambda_3>-\lambda_3>-\lambda_2>-\lambda_1=-1$$ 
\end{remark}

\subsection{Arithmeticity of $\rho(SL(\mathcal{O}_1))$ in $Sp(H_1^{(0)}(\mathcal{O}_1,\mathbb{R}))$}   

Denote by 
$$\Theta = \left(\begin{array}{cccc} 
1 & 1 & 1 & -1 \\ -1 & 0 & 0 & 1 \\ -1 & -1 & 0 & -1 \\ 0 & 1 & -1 & 1 
\end{array}\right)$$

After using $\Theta$ to change the basis $\mathcal{B}_1^{(0)}$, we obtain the matrices 
$$A:=\Theta^{-1}\rho(a)\Theta = \left(\begin{array}{cccc} 
 1 & 0 & 0 & 0 \\ 0 & 1 & 0 & 0 \\ 0 & 0 & -1 & 1 \\ 0 & 0 & -1 & 0
\end{array}\right), 
\quad 
B:=\Theta^{-1}\rho(b)\Theta = \left(\begin{array}{cccc}
 -1 & 0 & 0 & -1 \\ 
 0 & 0 & -1 & 0 \\ 
 0 & 1 & -1 & 0 \\ 
 1 & 0 & 0 & 0 
\end{array}\right)$$

\begin{remark} Since the matrices $\rho(a)$ and $\rho(b)$ preserve the symplectic form induced by $\Omega$ in Subsection \ref{ss.absolute-homology} above, we have that $\Theta^{-1}\rho(a)\Theta$ and $\Theta^{-1}\rho(b)\Theta$ are symplectic matrices with respect to 
$$\Theta^t\Omega\Theta = \left( 
\begin{array}{cccc} 
0 &-9 & 0 & 0 \\ 9 & 0 & 0 & 0 \\ 0 & 0 & 0 & 9 \\ 0 & 0 & -9 & 0 
\end{array} \right)$$
\end{remark}

At this point, the proof of Theorem \ref{t.HM} is reduced to: 

\begin{theorem}\label{t.HM'} $\rho(\textrm{Aff}(\mathcal{O}_1))$ has finite index in $\textrm{Sp}(H_1^{(0)}(\mathcal{O}_1),\mathbb{Z})$.
\end{theorem}

\begin{proof}
Let us consider the permutation matrix
$$P =  \left( 
\begin{array}{cccc} 
1 & 0 & 0 & 0 \\ 0 & 0 & 0 & 1 \\ 0 & 0 & 1 & 0 \\ 0 & 1 & 0 & 0 
\end{array} \right)$$ 
exchanging the second and fourth basis vectors and let us show that the conjugate 
$$P\cdot\langle A,B\rangle\cdot P$$ of $\rho(\textrm{Aff}(\mathcal{O}_1))=\langle A,B\rangle$ is arithmetic, i.e., it has finite-index in $Sp(4,\mathbb{Z})$. 

We found\footnote{For this sake, we asked Sage to look words on $A$, $B$, $A^2$ and $B^2$ of size $\leq 10$ fixing the first basis vector.} that the matrices $x=P(A^2B)^2(AB^2)^2P$, $y=PABA^2BA(AB^2)^2P$ and $z=PA^2BA^2(B^2A)^2BP$ are interesting because
$$[y,x]=yxy^{-1}x^{-1} = \left( 
\begin{array}{cccc} 
1 & 0 & 0 & 18 \\ 0 & 1 & 0 & 0 \\ 0 & 0 & 1 & 0 \\ 0 & 0 & 0 & 1 
\end{array} \right), \quad 
x^6[y,x] = \left( 
\begin{array}{cccc} 
1 & 0 & 18 & 0 \\ 0 & 1 & 0 & 18 \\ 0 & 0 & 1 & 0 \\ 0 & 0 & 0 & 1 
\end{array} \right),$$ 
$$y^6[y,x]^{-1} = \left( 
\begin{array}{cccc} 
1 & 18 & 0 & 0 \\ 0 & 1 & 0 & 0 \\ 0 & 0 & 1 & -18 \\ 0 & 0 & 0 & 1 
\end{array} \right), \quad 
z^6(x^6[y,x])^{-1} = \left( 
\begin{array}{cccc} 
1 & 0 & 0 & 0 \\ 0 & 1 & -18 & 0 \\ 0 & 0 & 1 & 0 \\ 0 & 0 & 0 & 1 
\end{array} \right)$$
generate the positive root groups of $Sp(4,\mathbb{R})$ and, thus, $P\cdot \langle A,B\rangle\cdot P$ intersects the subgroup $U(\mathbb{Z})$ of unipotent upper triangular matrices of $Sp(4,\mathbb{Z})$ in a finite-index subgroup\footnote{This argument was inspired by Section 2 of Singh and Venkataramana paper \cite{SV}. Note that if we want to generate a finite-index subgroup of the unipotent \emph{radical} of the parabolic subgroup associated to the flag $\mathbb{Q}e_1\subset \mathbb{Q}e_1\oplus \mathbb{Q}e_2\oplus \mathbb{Q}e_3\subset \mathbb{Q}^4$, then it suffices to use the matrices $[y,x]$, $x^6[y,x]$ and $y^6[y,x]^{-1}$.}.

Since we know that $\langle A,B\rangle$ is Zariski-dense, we can apply the arithmeticity criterion of Oh \cite{Oh} and Benoist--Miquel \cite{BM} saying that Zariski dense subgroups of $Sp(4,\mathbb{Z})$ containing a finite-index subgroup of $U(\mathbb{Z})$ are arithmetic to get the desired conclusion.
\end{proof}

\begin{remark}\label{r.Kazhdan} S. Kohl pointed out\footnote{Actually he computed with GAP the words on $A, B, A^{-1}, B^{-1}$ of sizes $1, 2, \dots$, and he noticed that the set of words of length $12$ has size $<2^{12+1}$. This led him to the nontrivial relation of length $2\cdot12=24$ above.} to us that $\rho$ is not faithful: indeed, 
$$(ABA^{-1}BA^{-1}BAB^{-1})^3 = \textrm{Id},$$ 
so that $(aba^{-1}ba^{-1}bab^{-1})^3 = \left( 
\begin{array}{cc} -24587 & 42408 \\ 15048 & -25955 \end{array} \right)$ lies in $\textrm{ker}(\rho)$. 

This is coherent with the arithmeticity statement in Theorem \ref{t.HM'}: if $\rho$ were faithful, then $\textrm{Sp}(H_1^{(0)}(\mathcal{O}_1),\mathbb{Z})$ would contain a finite-index subgroup isomorphic to a free group\footnote{Alternatively, Lemma \ref{l.Veech-group-structure} could be directly used to show that $\rho(\textrm{Aff}(\mathcal{O}_1))$ would contain a finite-index subgroup isomorphic to a free group if $\rho$ were faithful.} on five generators, namely $\rho(\Gamma(4))\subset \rho(\textrm{Aff}(\mathcal{O}_1))$. This is a contradiction because it is well-known that $\textrm{Sp}(4,\mathbb{Z})$ does not contain lattices isomorphic to free groups (thanks to Kazhdan property (T)).
\end{remark}

\subsection{Final comments} 

This note grew up from the following attempt to produce examples of origamis generating thin Kontsevich--Zorich monodromies. 

By an argument in the spirit of Remark \ref{r.Kazhdan}, if $\mathcal{O}$ is an origami of genus $g\geq 3$ such that the representation $\rho:\textrm{Aff}(\mathcal{O})\to \textrm{Sp}(H_1^{(0)}(\mathcal{O},\mathbb{Z}))$ is faithful and $\rho(\textrm{Aff}(\mathcal{O}))$ is Zariski-dense in $\textrm{Sp}(H_1^{(0)}(\mathcal{O},\mathbb{R}))$, then $\mathcal{O}$ has thin Kontsevich--Zorich monodromy. On the other hand, if $\mathcal{O}$ has some direction with homological dimension one (i.e., whose cylinders have waist curves spanning a one-dimensional subspace of $H_1(\mathcal{O},\mathbb{R})$), then it is not hard to check that a Dehn multitwist along this direction would belong to the kernel of $\rho$. Hence, it is natural to try to detect origamis with thin Kontsevich--Zorich monodromies among the origamis without directions of homological dimension one. 

\begin{remark} A related strategy towards the same goal would be to show that $\rho(\textrm{Aff}(\mathcal{O}))$ fits the assumptions of the ping-pong lemma (compare with the proof of Lemma \ref{l.Veech-group-structure}). Nevertheless, it is not easy to implement this idea in general because the construction of ``ping-pong subsets'' might be somewhat tricky (see page 5387 and also Subsections 3.2 and 3.3 of Fuchs--Rivin paper \cite{FR}). 
\end{remark}

As it turns out, the origami $\mathcal{O}_1$ is one of the smallest examples of origamis of genus $3$ having no direction with homological dimension one (compare with Remark \ref{r.homological-dim-O1}) and this explains our interest on its Kontsevich--Zorich monodromy. 

Anyhow, once we detect a good candidate origami $\mathcal{O}$, the first step is the computation of its Kontsevich--Zorich monodromy, i.e., the Zariski closure of $\rho(\textrm{Aff}(\mathcal{O}))$ (compare with Subsection \ref{ss.Zariski}). Here, the criterion of Prasad--Rapinchuk \cite[Theorem 9.10]{PR} (see also \cite[Theorem 1.5]{Ri}) informally says that the Zariski closure is ``often'' a symplectic group $Sp$ or a product of $SL_2$'s. Moreover, the techniques in \cite{MMY} indicate that the Zariski closure tends to be a symplectic group in many situations including $\mathcal{H}(4)$, but this must be taken with a grain of salt because the case of products of $SL_2$ happens in nature: for instance, Eskin--Kontsevich--Zorich \cite{EKZ2} noted that the so-called ``stairs'' origamis in $\mathcal{H}(2g-2)$ and $\mathcal{H}(g-1,g-1)$ are covered by special ``square-tiled cyclic covers'' and this information can be used to show that the Kontsevich--Zorich monodromy of a ``stairs'' origami is contained\footnote{Actually, we did some computations with the \emph{first few} stairs origamis and their Kontsevich--Zorich monodromies turned out to be \emph{equal} to products of $SL_2$.} in a product of $SL_2$'s. 

Finally, even if $\rho(\textrm{Aff}(\mathcal{O}))$ is Zariski-dense in $\textrm{Sp}(H_1^{(0)}(\mathcal{O},\mathbb{R}))$, it is certainly a challenging problem to obtain the faithfulness of $\rho$. Here, the case of arithmetic Teichm\"uller curves of genus zero might be a good starting point of investigation (because the corresponding Veech groups are generated by elliptic and parabolic elements of $SL(2,\mathbb{Z})$), but our discussion of $\mathcal{O}_1$ in the previous subsection shows that this situation is not always favourable towards the construction of thin Kontsevich--Zorich monodromies. 


\end{document}

%% file: thinkz1.pdf_tex

\begingroup
  \makeatletter
  \providecommand\color[2][]{%
    \errmessage{(Inkscape) Color is used for the text in Inkscape, but the package 'color.sty' is not loaded}
    \renewcommand\color[2][]{}%
  }
  \providecommand\transparent[1]{%
    \errmessage{(Inkscape) Transparency is used (non-zero) for the text in Inkscape, but the package 'transparent.sty' is not loaded}
    \renewcommand\transparent[1]{}%
  }
  \providecommand\rotatebox[2]{#2}
  \ifx\svgwidth\undefined
    \setlength{\unitlength}{350pt}
  \else
    \setlength{\unitlength}{\svgwidth}
  \fi
  \global\let\svgwidth\undefined
  \makeatother
  \begin{picture}(1,0.42753835)%
    \put(0,0){\includegraphics[width=\unitlength]{thinkz1.pdf}}%
    \put(0.48960331,0.34465063){\color[rgb]{0,0,0}\makebox(0,0)[lb]{\smash{$1$}}}%
    \put(0.48960331,0.20600434){\color[rgb]{0,0,0}\makebox(0,0)[lb]{\smash{$6$}}}%
    \put(0.48960331,0.06735458){\color[rgb]{0,0,0}\makebox(0,0)[lb]{\smash{$3$}}}%
    \put(0.62822867,0.20607678){\color[rgb]{0,0,0}\makebox(0,0)[lb]{\smash{$7$}}}%
    \put(0.76683953,0.20607678){\color[rgb]{0,0,0}\makebox(0,0)[lb]{\smash{$8$}}}%
    \put(0.0736395,0.06782632){\color[rgb]{0,0,0}\makebox(0,0)[lb]{\smash{$4$}}}%
    \put(0.35097287,0.06752334){\color[rgb]{0,0,0}\makebox(0,0)[lb]{\smash{$2$}}}%
    \put(0.21236198,0.06752334){\color[rgb]{0,0,0}\makebox(0,0)[lb]{\smash{$5$}}}%
    \put(0.90535676,0.20608861){\color[rgb]{0,0,0}\makebox(0,0)[lb]{\smash{$9$}}}%
  \end{picture}%
\endgroup

%% file: thinkz2.pdf_tex

\begingroup
  \makeatletter
  \providecommand\color[2][]{%
    \errmessage{(Inkscape) Color is used for the text in Inkscape, but the package 'color.sty' is not loaded}
    \renewcommand\color[2][]{}%
  }
  \providecommand\transparent[1]{%
    \errmessage{(Inkscape) Transparency is used (non-zero) for the text in Inkscape, but the package 'transparent.sty' is not loaded}
    \renewcommand\transparent[1]{}%
  }
  \providecommand\rotatebox[2]{#2}
  \ifx\svgwidth\undefined
    \setlength{\unitlength}{250pt}
  \else
    \setlength{\unitlength}{\svgwidth}
  \fi
  \global\let\svgwidth\undefined
  \makeatother
  \begin{picture}(1,0.85448358)%
    \put(0,0){\includegraphics[width=\unitlength]{thinkz2.pdf}}%
    \put(0.23,0.19644001){\Huge\color[rgb]{0,0,0}\makebox(0,0)[b]{\smash{$\mathcal{O}_2$}}}%
    \put(0.67,0.19644001){\Huge\color[rgb]{0,0,0}\makebox(0,0)[b]{\smash{$\mathcal{O}_3$}}}%
    \put(0.23,0.63957576){\Huge\color[rgb]{0,0,0}\makebox(0,0)[b]{\smash{$\mathcal{O}_1$}}}%
    \put(0.67,0.63957576){\Huge\color[rgb]{0,0,0}\makebox(0,0)[b]{\smash{$\mathcal{O}_4$}}}%
    \put(0.3731932,0.41321498){\color[rgb]{0,0,0}\makebox(0,0)[b]{\smash{$S$}}}%
    \put(0.53529357,0.41321498){\color[rgb]{0,0,0}\makebox(0,0)[b]{\smash{$S$}}}%
    \put(0.44938779,0.50062647){\color[rgb]{0,0,0}\makebox(0,0)[b]{\smash{$S$}}}%
    \put(0.44938779,0.31767905){\color[rgb]{0,0,0}\makebox(0,0)[b]{\smash{$S$}}}%
    \put(0.88108037,0.40806651){\color[rgb]{0,0,0}\makebox(0,0)[b]{\smash{$T$}}}%
    \put(0.03634049,0.40806651){\color[rgb]{0,0,0}\makebox(0,0)[b]{\smash{$T$}}}%
    \put(0.45495578,0.81821426){\color[rgb]{0,0,0}\makebox(0,0)[b]{\smash{$T$}}}%
    \put(0.45495578,0.00370047){\color[rgb]{0,0,0}\makebox(0,0)[b]{\smash{$T$}}}%
  \end{picture}%
\endgroup